\documentclass{birkjour}
\usepackage[utf8]{inputenc}
\newcommand{\mM}{\mathcal{M}}

\usepackage{amsmath,amssymb,amsthm}
\usepackage{color}
\usepackage{graphicx}
\usepackage{subcaption}

\newtheorem{theorem}{Lemma}

\title{The conjugate locus in convex 3-manifolds}
\author{Thomas Waters and Matthew Cherrie}

\begin{document}

\maketitle

\begin{abstract}
In this paper we study the conjugate locus in convex manifolds. Our main tool is Jacobi fields, which we use to define a special coordinate system on the unit sphere of the tangent space; this provides a natural coordinate system to study and classify the singularities of the conjugate locus. We pay particular attention to 3-dimensional manifolds, and describe a novel method for determining conjugate points. We then make a study of a special case: the 3-dimensional (quadraxial) ellipsoid. We emphasise the similarities with the focal sets of 2-dimensional ellipsoids.
\end{abstract} 


\section{Introduction}

The conjugate locus is a classical topic in global Differential Geometry, and one of the driving open problems has been the `last geometric statement of Jacobi': that the conjugate locus of a generic point on a triaxial ellipsoid has four cusps \cite{jacobi}. The question was settled recently by Itoh and Kiyohara, who also showed the statement extends to the broader class of surfaces known as Liouville surfaces \cite{Itoh1,Itoh2}. It seems natural to ask how these results extend to higher dimensions, and indeed Itoh and Kiyohara in \cite{Itoh3,Itoh4} have considered the cut and conjugate locus in $n$-dimensional ellipsoids and Liouville manifolds. Naturally their analysis relies heavily on the integrability of the geodesic flow in these manifolds, however integrability is the exception rather than the norm; see \cite{TWspherical, TWmonge, TWetds} for example. A guiding principle in this and other works of the authors is to develop methods for studying the conjugate locus in manifolds {\it without} relying on integrability.  

The conjugate locus, and its cousin the cut locus, has been studied by some of the giants of Mathematics: Jacobi \cite{jacobi}, Poincar\'{e} \cite{poincare} and Arnol'd \cite{arnold} to name a few. As well as investigations related to the Jacobi statement (see for example \cite{Sinclair1, Sinclair2, Sinclair3, Sinclair4, Sinclair5}) and a testbed for developments in Hamiltonian dynamics (see for example \cite{mcoff1,mcoff2}), there is also interest in the application to problems in optimal control (see for example \cite{Bonnard1, Bonnard2, Bonnard3, Bonnard4, Bonnard5, bloch}). Previous work of the first author showed how the number of cusps can sponaneously change as the base point is dragged around in the surface \cite{TWbif}, and the connection between the number of cusps and the rotation index \cite{TWrot}. However (aside from \cite{Itoh4,mcoff1} mentioned above and perhaps \cite{hannes}) there is very little in the literature regarding the conjugate locus in dimensions three or higher; a notable exception being Ardoy \cite{ardoy} which showed the generic singularities of the exponential map in 3-manifolds come in five varieties: fold, edge, swallowtail, elliptic and hyperbolic umbilic. We note that these are also the five generic singularities of the normal map, i.e.\ the envelope of normals to surfaces in $\mathbb{R}^3$ (also known as the evolute or focal set). There is much more in the literature about these focal sets, in particular see \cite{porteous,izumiya}, and the focal sets of surfaces are typcially made up of two sheets of focal points. A classic example would be the focal sets of the triaxial ellipsoid (or 2-ellipsoid as we will refer to it), whose picture is strangely absent from the literature (but see \cite{jandr}). We present in this work for the first time global images of the conjugate locus in the 3-ellipsoid, and the similarities with the focal sets of the 2-ellipsoid are striking; indeed a recurring theme in this paper (and others of the first author \cite{TWrot}) would be the similarity between the focal set on one hand and the conjugate locus on the other. This is continuing the train of thought in Jacobi's work, who originally framed his famous statement by saying the conjugate locus on the triaxial ellipsoid has ``die gestalt der evolute der ellipse'' \cite{jacobi}. We suggest the following extension to Jacobi's statement:

\medskip

\noindent {\it Suppose $p$ is a point in an $n$-dimensional ellipsoid with principal curvatures $k_1,\ldots,k_n$. Then the conjugate locus of $p$ has the same form as the focal sets of the $(n-1)$-dimensional ellipsoid with semi-axes $k_1,\ldots,k_n$.} 

\medskip

This conjecture was presented by the first author at Laboratoire Dieudonn\'{e}, Nice, in January 2019, several months before the publication of \cite{Itoh4}, and we hope the current work complements \cite{Itoh4} while extending our understanding of this most important of example manifolds. We note that by ``the same form'' we mean the same number of components, the same singularity structure, even the same topology in the sense of rotation index. For example when $n=2$ we know the conjugate locus of an umbilic point in the triaxial ellipsoid is a point (as is the evolute of a circle), and the conjugate locus of a generic point is a single curve with 4 cusps and rotation index 1 (as is the evolute of an ellipse).

As well as making a study of the 3-ellipsoid, the main contribution of this paper lies in developing methods for studying the conjugate locus in convex smooth 3-dimensional Riemannian manifolds in general (note we make the restriction of convexity, by which we mean positive sectional curvatures in all directions everywhere, so that conjugate points exist in all directions; however this restrictions is too strong if we simply want a non-empty conjugate locus, as the spherical harmonic surfaces of \cite{TWbif} demonstrate). Other authors relied on integrability \cite{Itoh4} or contact \cite{porteous} however for convex 3-manifolds in general we cannot rely on either of those things, and so we describe a novel approach in Sections 2 and 3: Section 2 using Jacobi fields to define a new coordinate system on the unit sphere in $T_p\mathcal{M}$, and Section 3 describing a novel approach to detecting conjugate points. We feel the additional understanding gained justifies our choice of Jacobi fields as the main tool rather than the marching fronts of \cite{hannes} or bifurcations of a Hamiltonian boundary value problem in \cite{mcoff2}. In Section 4 we make a detailed case study of the 3-dimensional ellipsoid, and we finish with some conclusions and suggestions for future work.

We follow the conventions of DoCarmo \cite{docarmoriemm}, and we let $t$ be arc-length parameter and $'$ denote $d/dt$. We are most interested in 3-dimensional manifolds, but in Section 2 we begin with $n$-dimensions.

\section{The Jacobi coordinate system}

Let $p$ be a point in the $n$-dimensional smooth Riemannian manifold $\mM$ and split $T_p \mM$ as $\mathbb{R}^+\times\mathbb{S}^{n-1}_1$. Let $U$ be a coordinate patch on $\mathbb{S}^{n-1}_1$ and let $v_i$ (where from now on $i,j$ will run from 1 to $n-1$) parameterize $U$; we will assume $U$ is `generic' in the sense to be clarified below. We define the exponential map at $p$ as (it is convenient to use ``$X$'' rather than ``$exp$'') \[ X:\mathbb{R}^+\times U\to\mM : X(t,v_i)=\gamma(t) \] where $\gamma(t)$ is the geodesic through $p$ whose tangent vector at $p$ is the point in $U$ with coordinates $v_i$ (see Figure \ref{jacocoords} below).

Conjugate points are where the exponential map is singular, but since $\partial X/\partial t=\gamma'(t)$ the co-rank of $DX$ is at most $n-1$. We define the Jacobi fields $Y_i$ as $\partial X/\partial v_i$, and since we are interested in conjugate points we will from now on only consider Jacobi fields which meet the following conditions: all $Y_i$ have $Y_i(t=0)=0$, and $Y'_i(t=0)$ (what distinguishes one Jacobi field from another) is taken from the orthogonal complement of $\gamma'(0)$. As we follow a radial geodesic emanating from $p$ we may reach points where a non-trivial Jacobi field vanishes for the first time (by which we mean this is the first time since $t=0$ that that particular Jacobi field has vanished), and these are the `first conjugate points'; the set of all first conjugate points is the `first conjugate locus' which may have many sheets as we will see. We will assume for the time being that $\mathcal{M}$ is such that for all radial geodesics in $U$, all the conjugate points occur at distinct values of $t$; this is what we meant by `generic' above (this is equivalent to saying that all singularities of $X$ over $U$ are co-rank 1, or $\Sigma^1$ to use the notation of \cite{jandr}). 








We will now use these conjugate points to define a special set of coordinates on $U$. Let $R_1,R_2,\ldots,R_{n-1}$ be the ordered distance along a radial geodesic until a first conjugate point. Since they are distinct, each corresponds to the vanishing of a distinct member of a set of $n-1$ linearly independent Jacobi fields, which we denote $J_i$, distinguished by $n-1$ linearly independent choices of $J_i'(t=0)$, which span the orthogonal complement of $\gamma'(t=0)$. We now let $u_i$ be the coordinate on $U$ whose coordinate lines have tangent vector $J_i'(t=0)$ (see Figure \ref{jacocoords}). This is a smooth parameterization of $U$: if $U$ is generic then at every point in $U$ there are $n-1$ linearly independent $J_i'(0)$ tangent to $\mathbb{S}^{n-1}_1$ (the $u_i$ coordinate lines are their integral curves), and for smoothness we note that by definition $J_i(t=R_i(u),u)=0$ and differentiating w.r.t.\ $u_j$ we find \[ J_i'\frac{\partial R_i}{\partial u_j}=-\frac{\partial J_i}{\partial u_j}. \] The right hand side is smooth w.r.t.\ $u$ since Jacobi fields are the solutions of ODE's whose coefficients depend smoothly on $u$; therefore the left hand side is also smooth w.r.t.\ $u$. These special coordinates, which we will refer to as `Jacobi coordinates', parameterize both the patch $U$ on the unit sphere in $T_p\mM$ and the conjuate locus itself, and provide a natural coordinate system in which to describe the structure of the conjugate locus; some results are collected as Lemmas:

\begin{figure}
\begin{center}
\includegraphics[width=0.7\textwidth]{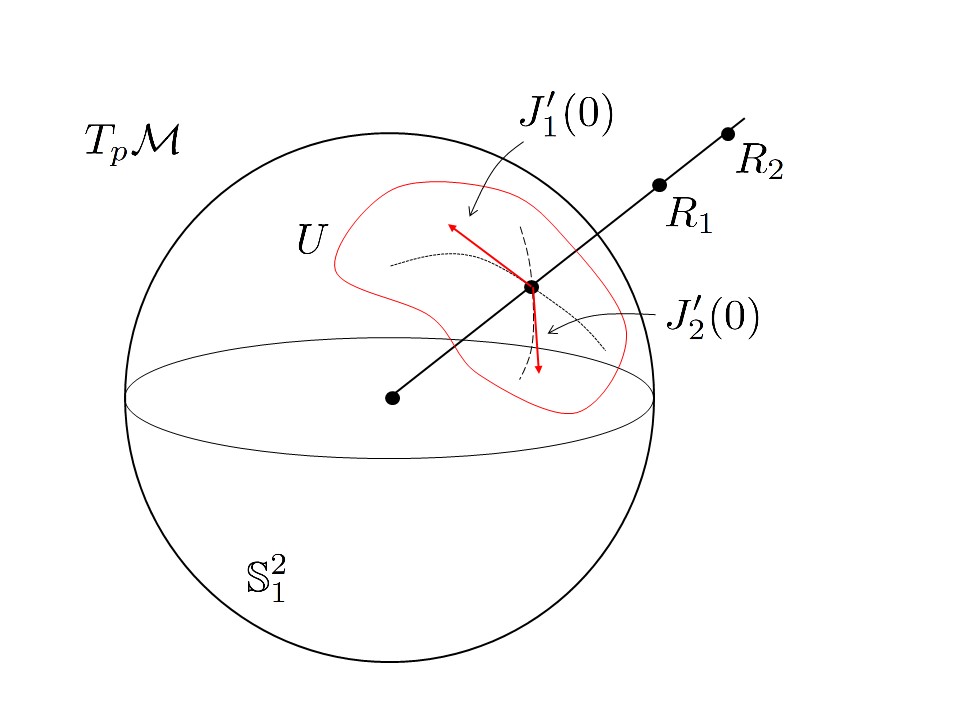}
\caption{The Jacobi coordinates in the case that $\mM$ is 3 dimensional. The $u_1$ coordinate line ($u_2=$ constant) and $u_2$ coordinate line ($u_1=$ constant) are shown dotted and dashed respectively.}\label{jacocoords}
\end{center}
\end{figure}

\begin{theorem}
 The first conjugate locus is regular except for points where the distance to the $i$-th first conjugate point is stationary with respect to the $i$-th Jacobi coordinate.
\end{theorem}

\begin{proof} We define the $i$-th sheet of the first conjugate locus as \[ c_i=X(t=R_i(u),u). \] The Jacobian of the $i$-th sheet has rows \begin{align} \frac{\partial c_i}{\partial u_j}=\frac{\partial X}{\partial t}\frac{\partial R_i}{\partial u_j}+\frac{\partial X}{\partial u_j}=\gamma' \frac{\partial R_i}{\partial u_j}+J_j(t=R_i).\label{bian} \end{align} Since $\gamma'$ and $J_j$ are linearly independent, we see that each sheet of the conjugate locus is regular in general, however if $j=i$ then since $J_i(t=R_i)=0$ we see the rank of the Jacobian of $c_i$ will be non-maximal if $\partial R_i/\partial u_i=0$. We will refer to these $\Sigma^{1,1}$ singularities (when restricted to $U$) as `ridges' and their image under $X$ as `ribs', in analogy with the Euclidean case described in the introduction.\end{proof}

The most intuitive case (see Sections 3 and 4 below for examples) is when the manifold $\mM$ has dimension 3; then the ridge points form lines on the unit 2-sphere in $T_p\mM$. For the rest of this paper we will suppose $n=3$.

\begin{theorem} The ribs on the sheets of the conjugate locus are ordinary cuspidal edges.\end{theorem}

\begin{proof} If we label the two Jacobi coordinates $u_1=u$ and $u_2=v$, with corresponding Jacobi fields $J_u$ and $J_v$, then the first sheet of the first conjugate locus is given by $c_1=X(t=R_1(u,v),u,v)$. The image under $X$ of the ridge line $R_{1,u}=0$ will be an ordinary cuspidal edge if \cite{porteous} \[ c_{1,u}=0,\qquad \{ c_{1,uu},c_{1,uuu},c_{1,v}\}\ \textrm{linearly independent.} \] We have already seen that at a ridge point $c_{1,u}=0$, and under the same condition we find \begin{align*}
c_{1,uu}&=R_{1,uu}X_{,t}=R_{1,uu}\gamma',  \\ c_{1,uuu}&=2R_{1,uu}X_{,tu}+R_{1,uuu}X_{,t}=2R_{1,uu}J_u'+R_{1,uuu}\gamma', \\ c_{1,v}&=R_{1,v}X_{,t}+X_{,v}=R_{1,v}\gamma'+J_v
\end{align*} where everything is evaluated on $t=R_1$. Since the Jacobi fields and their derivatives lie in the orthogonal complement of $\gamma'$, and since $J_v(t=R_1)\neq 0$ in a generic coordinate patch, then assuming higher derivatives of $R_1$ do not vanish simultaneously with the first we see these three vectors will be linearly indepenent if $J_u'$ and $J_v$ are linearly independent on $t=R_1$. This is certainly the case if $\mathcal{M}=\mathbb{S}^3$ (indeed they are orthogonal), but in general we would need to show linear independence. Both vector fields are non-zero: we already know $J_v$ cannot vanish on $t=R_1$, but also $J_u'$ cannot vanish on $t=R_1$ since $J_u$ satisfies a second order homogeneous ODE and we already have $J_u(t=R_1)=0$. Finally we must show that: \begin{align} J_u' \ \textrm{and}\ J_v\ \textrm{are not multiples of each other}. \label{cuspstat} \end{align} We will see in the next section (see \eqref{j1j2det}) how the method we describe for detecting conjugate points answers this question very naturally.\end{proof}

\begin{theorem} Away from ridge points, the image of the Jacobi coordinate lines under the exponential map are geodesics of the corresponding sheet of the conjugate locus.\end{theorem}

\begin{proof} If we take the first sheet for example, then using \eqref{bian} the line element for $c_1$ is \[ ds^2=R_{1,u}^2du^2+2R_{1,u}R_{1,v}dudv+(R_{1,v}^2+|J_2(t=R_1)|^2)dv^2, \] or (using $dR_1=R_{1,u}du+R_{1,v}dv$) we can write this as \[ ds^2=dR_1^2+|J_2(t=R_1)|^2dv^2. \] Now it is standard to show that the lines $v=$ constant (the image under $X$ of the $u$ Jacobi coordinate lines) are geodesics of the first sheet of the conjugate locus. \end{proof}

Previously we assumed $U$ only contained points defining directions along which the distances to first conjugate points were distinct, but if we consider all directions on the unit sphere in $T_p\mathcal{M}$ there will be some directions where $R_1=R_2$, and we refer to these $\Sigma^2$ singularities as `umbilic directions', again in analogy with the Euclidean case. That these umbilic directions must exist follows from the definition of the Jacobi coordinate system: if there were no umbilic directions then we would be able to cover the 2-sphere in a single coordinate system which we know is not possible (it seems natural to suppose an extenison of the Caratheodory conjecture \cite{guilf} in this context). We are starting to see the general picture emerge: the Jacobi coordinate system covers the unit sphere except at degenerate umbilic directions; $R_1$ and $R_1$ are functions over the sphere which must therefore have stationary points through which pass the ridge lines; if the Jacobi coordinate curves are closed then $R_1$ or $R_2$ restricted to them must have stationary points which correspond to ridge lines of different types, and so on.

There are many ways to view the objects described in this section: in Section 4 we show plots of the first conjugate locus in the 3-ellipsoid, but another option is to plot the `distance spheres', i.e.\ the polar surfaces $r=R_1$ and $r=R_2$ in $T_p\mathcal{M}$; these would sit one inside the other, and touch along the umbilic directions. It is the image of these distance spheres under the exponantial map which give the first conjugate locus, however it is difficult to see the ridges on these distance spheres (unlike the vertices of the distance curves when $n=2$, see for example \cite{myers},\cite{TWbif}). Another possibility is to define on the unit sphere of $T_p\mathcal{M}$ the function $1/(R_1R_2)$, a sort of `global Gauss curvature'. Finally, as the point $p$ is dragged around in $\mathcal{M}$ we would expect the ridge lines on the unit sphere in $T_p\mathcal{M}$ (see Fig \ref{JC0} for example) to deform and perhaps undergo spontaneous transformations, of the type described in \cite{TWbif} for $n=2$ and \cite{porteous},\cite{soto} for the Euclidean setting; but for now we turn our attention to detecting conjugate points and hence the Jacobi coordinates and ridge lines described in this section.

\section{Detecting conjugate points}

The key question in this section is: how do we detect conjugate points along a unit speed geodesic $\gamma$ emanating from some point $p$ in $\mathcal{M}$? This question is reasonably straightforward in dimension 2: we simply solve the following initial value problem \[ \xi''+K|_\gamma\, \xi=0,\quad \xi(0)=0,\ \xi'(0)=1 \] (where $K$ is the Gauss curvature) until $\xi=0$; in a convex manifold this is guaranteed to happen within $t\leq \pi/\sqrt{K_{min}}$. However the complication in manifolds of dimension 3 is that there is now enough `room' for neighbouring geodesics to pass {\it behind} $\gamma$ (conjugate points are typically defined as we have above in terms of the vanishing of Jacobi fields, but the more intuitive picture of intersecting neighbouring geodesics can be made formal in the variational framework \cite{gelf}). Our problem is that not all orthogonal Jacobi fields will vanish at a conjugate point, and we do not know {\it a priori} which Jacobi fields to track. What we seek is some scalar along $\gamma$ that vanishes (or even better changes sign) at a conjugate point. We emphasise that the methods we wish to develop are to apply to convex 3-manifolds in general, without assuming integrability or symmetry for example.

Let $\gamma$ be a unit speed geodesic emanating from $p\in\mathcal{M}$ and let $\{\gamma'=T,N,B\}$ be an orthonormal triple parallel transported along $\gamma$ (in general there is no {\it a priori} choice for $N$ and $B$ since $\gamma'$ is an arbitrary direction in $T_p\mathcal{M}$). We write a Jacobi field along $\gamma$ as \[ J=\beta T+ \xi N+\eta B. \] All Jacobi fields along $\gamma$ satisfy the same Jacobi equation
\begin{align}
    \frac{D^{2}J}{dt^{2}}+R(\gamma'(t), J)\gamma'(t) = 0, \label{jacobieq}
\end{align} what distinguishes one Jacobi field from another is the initial conditions; as we are interested in detecting points conjugate to $p$ we set $\beta\equiv 0$. This leads to the linear system \begin{align}
   \begin{pmatrix}\xi'' \\ \eta''\end{pmatrix}=-\begin{pmatrix}
    (T,N,T,N) & (T,N,T,B) \\
    (T,N,T,B) & (T,B,T,B)
    \end{pmatrix} \begin{pmatrix}\xi \\ \eta\end{pmatrix} \label{jacobiequations}
\end{align} where $(A,B,C,D)=\left\langle R(A,B)C,D\right\rangle$. This symmetric time-dependent matrix has sectional curvatures on the diagonal and contractions of the Ricci tensor on the off-diagonal.

We now distinguish solutions to \eqref{jacobiequations} by their initial conditions. The constraint of $J(0)=0$ implies that $\xi(0)=\eta(0)=0$, leaving us with a 2-d space of initial conditions in which to search for conjugate points. Due to the linearity of \eqref{jacobiequations} we are able to rescale the $(\xi'(0),\eta'(0))$ initial conditions to lie on the unit circle, to arrive at a 1-parameter space of admissable Jacobi fields. To realise this we define two linearly independent Jacobi Fields, $J_\xi$ and $J_\eta$, whose initial conditions are given by
\begin{align}   
     \{\xi(0)=\eta(0)=0,\xi'(0)&=1, \eta'(0)=0 \}, \label{jxi} \\
    \{\xi(0)=\eta(0)=0,\xi'(0)&=0, \eta'(0)=1 \} \label{jeta}
\end{align} respectively, and the follwing linear combination captures the family of Jacobi field of interest:
\begin{align}
    {J}_{\alpha}={J}_{\xi}\cos{\alpha}+{J}_{\eta}\sin{\alpha},\ \alpha \in [0,2\pi]. \label{nbellipse}
\end{align} For any value of $t$, $J_{\alpha}$ is an ellipse in the $\{N,B\}$ plane, with each point on the ellipse corresponding to a specific Jacobi field, growing and shrinking and rotating as we move along $\gamma$. When we reach a conjugate point a Jacobi field vanishes, meaning that $J_{\alpha}$ collapses to a line segment at that value of $t$. If $J_\xi=\xi_1 N+\eta_1 B$ and $J_\eta=\xi_2 N+\eta_2 B$, then we define 
\begin{align}
    [J_\xi,J_\eta]=\begin{vmatrix} \xi_1 & \xi_2 \\ \eta_1 & \eta_2 \end{vmatrix}. \label{ellipseareaeq}
\end{align} 
The area of the ellipse in the $\{N,B\}$ plane is proportional to $[J_\xi,J_\eta]$ which vanishes at a conjugate point. Thus, plotting this scalar versus $t$ allows us to find the values of $t=R_{1}$ and $t=R_{2}$. Even better, the ellipse must pass through itself (or change orientation) are we pass a conjugate point and so this scalar will change sign, which makes detecting zeroes easier. This is because conjugate points must be simple zeroes of Jacobi fields (we cannot have a non-trivial $J$ which satisfies both $J=0$ and $J'=0$ at some point), and so the vanishing Jacobi field must change sign at a conjugate point and thus the ellipse changes orientiation: the area of the ellipse has a simple zero at a conjugate point (the exception to this is at an umbilic direction). This is made more clear if we derive an equation for the area of the ellipse as follows: if $J_1$ and $J_2$ are arbitrary Jacobi fields then \begin{align}
\frac{d}{dt}[J_1,J_2]=[J_1',J_2]+[J_1,J_2'], \label{j1j2det}\\ \frac{d^2}{dt^2}[J_1,J_2]=-Tr(M)[J_1,J_2]+2[J_1',J_2'] \label{mdd}
\end{align} where $M$ is the matrix in \eqref{jacobiequations} and $Tr(M)=2Ric_\gamma(\gamma')$. Unfortunately due to the inhomogeneous term in \eqref{mdd} we cannot simply solve this scalar equation instead of \eqref{jacobiequations}, however it is also this term we gives non-trivial solutions given the trivial initial data $[J_1,J_2](0)=\tfrac{d}{dt}[J_1,J_2](0)=0$. Nonetheless we can use \eqref{j1j2det} to answer the question posed in \eqref{cuspstat} in the previous section  regarding $J_u'$ and $J_v$ when showing the ridges lift to ordinary cuspidal edges: if $J_1=J_u$ and $J_2=J_v$ then evaluating \eqref{j1j2det} at $t=R_1$ (assumed distinct from $R_2$) we see \begin{align} \frac{d}{dt}[J_u,J_v]\bigg\vert_{t=R_1}=([J_u',J_v]+[J_u,J_v'])\bigg\vert_{t=R_1}=[J_u',J_v]\bigg\vert_{t=R_1}\neq 0
\end{align} since the area of the ellipse has a simple zero at a non-umbilic conjugate point and therefore $J_u'$ and $J_v$ are linearly independent at $t=R_1$. 

When $\mathcal{M}=\mathbb{S}^3_r$ then the area of the ellipse is a multiple of $\sin^2(t/r)$, with a double root at $t=\pi r$, and this double root is the case for umbilic directions. However when we perturb away from the sphere then generically the graph of $[J_\xi,J_\eta]$ will dip below the axis giving two distinct roots, and these are the values $t=R_1$ and $t=R_2$. Now we know the locations of conjugate points along a specific geodesic emanating from $p$, we simply do the same for a spread of directions over the unit sphere in $T_p\mathcal{M}$, and we can then plot the sheets of the first conjugate locus $c_1$ and $c_2$; see the next section for some examples. 

Now we know the distance functions $R_1,R_2$ we can draw the Jacobi coordinate system as follows: for a particular geodesic emanating from $p$, we find the value of $\alpha$ where $|J_\alpha(R_1)|=0$, say $\alpha_1$, and this gives $J_u'(0)=\cos\alpha_1 N(0)+\sin\alpha_1 B(0)$. We then follow this tangent vector field across the unit sphere to give a Jacobi coordinate line; a similar process for $J_v'(0)=\cos\alpha_2 N(0)+\sin\alpha_2 B(0)$ where $\alpha_2$ is the value of $\alpha$ where $|J_\alpha(R_2)|=0$. Finally for the ridge points we simply evaluate $R_1$ (or $R_2$) along the $J_u'(0)$ (or $J_v'(0)$) Jacobi coordinate line just described, and pick out stationary values. In the next section we will implement the methods described in an important example: the 3-dimenisonal ellipsoid.

\section{Case study - 3d ellipsoid}

Our previous analysis was for convex 3-manifolds in general, and now we would like to consider a particular case: the 3-ellipsoid given by $ \frac{x_1^2}{a^2}+\frac{x_2^2}{b^2}+\frac{x_3^2}{c^2}+\frac{x_4^2}{d^2}=1$ with $a,b,c,d \in \mathbb{R}$ (also known as the quadraxial ellipsoid). This manifold has been studied in several works, for example \cite{sotomayor2014umbilic} describe lines of umbilic points, \cite{jandr} describe the the focal sets in $\mathbb{R}^4$, and \cite{dullin} study the integrability of the geodesic flow in the case of equal middle semi-axes. An immediate question is what parameterisation to use; in \cite{dullin} (and \cite{Itoh4}) the ellipsoidal/elliptic coordinate system is used, in \cite{jandr} the ambient coordinates are used, and in \cite{sotomayor2014umbilic} a combination of spherical polar and elliptic coordinates are used depending on the symmetry class of the ellipsoid. For this current work we will use the following spherical polar type parameterisation
\begin{align*}
   (a\sin{\theta}\sin{\phi}\cos{\psi},b\sin{\theta}\sin{\phi}\sin{\psi},c\sin{\theta}\cos{\phi},d\cos{\theta}),
\end{align*}
with $\theta,\phi \in (0,\pi),\ \psi \in (0,2\pi)$. 

To detect conjugate points along a geodesic we need to simultaneously solve the geodesic equations, the parallel transport equations for $N,B$, and the Jacobi equations in \eqref{jacobiequations} and for these we need the components of the Riemann curvature. In this parameterisation there is essentially one term:
\begin{align*}
   R_{1212}=\left(\sin ^2(\phi) \left(\frac{\cos ^2(\psi)}{a^2}+\frac{\sin ^2(\psi)}{b^2}\right)+\frac{\cos ^2(\phi)}{c^2}+\frac{\cot ^2(\theta)}{d^2}\right)^{-1},
\end{align*} and from this we have \begin{align*}
   R_{1313}= R_{1212} \sin ^2\phi, \quad
   R_{2323}=R_{1212} \sin ^2\theta \sin ^2\phi
\end{align*} and the other non-zero terms come from the symmetries of the Riemann curvature. All other terms are zero, and this markedly reduces computational time.

\begin{figure}
\begin{center}
\includegraphics[width=0.4\textwidth]{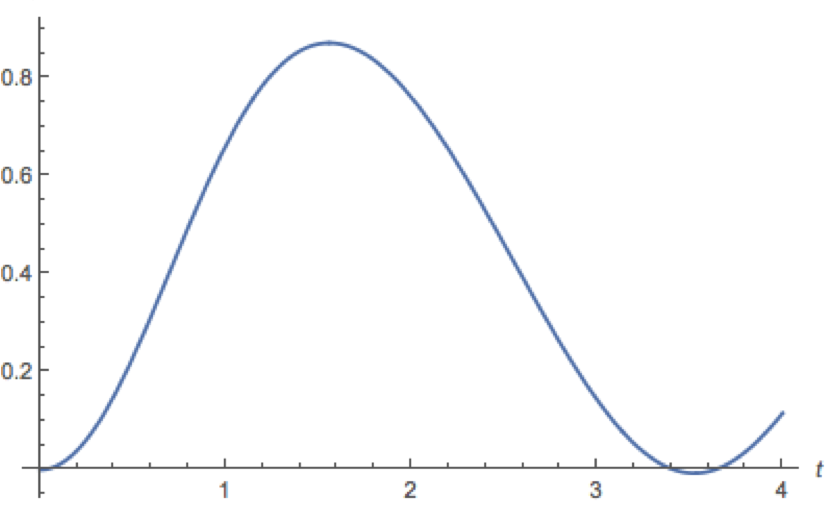}
\includegraphics[width=0.4\textwidth]{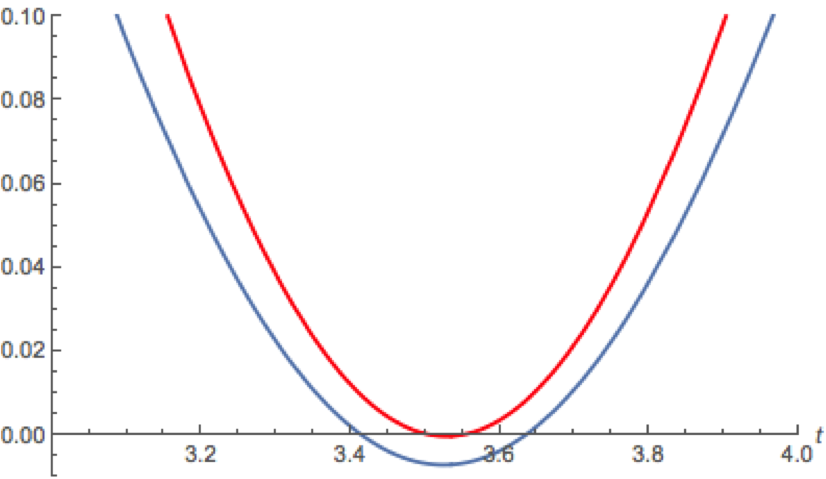}
\caption{The blue curve on the left (close up on the right) depicts the area of $J_\alpha$ plotted against $t$ for the initial conditions $(\dot{\theta},\dot{\phi},\dot{\psi})=(-0.730, 0.425, -0.774)$ and looks like a perturbed $\sin^{2}(t)$, as expected. The red curve depicts the area of $J_\alpha$ given by the initial conditions $(\dot{\theta},\dot{\phi},\dot{\psi})=(0.36,0.694,0.997)$, which touches the $t$-axis at a double-root.}\label{JArea}
\end{center}
\end{figure}

For demonstration purposes we will set the semi-axes to be $(a,b,c,d)=(0.9,1.05,1.15,1.2)$ and the base point $p$ to be $(\theta,\phi,\psi)=(\frac{\pi}{3},2.3,-\frac{\pi}{5})$. If we choose the tangent vector for a geodesic through $p$ from the unit sphere in $T_p\mathcal{M}$, we then choose $N(0),B(0)$ orthonormal to it and solve the geodesic equations, parallel transport equations and Jacobi equations \eqref{jacobiequations} twice with the initial conditions given in  \eqref{jxi},\eqref{jeta} to find $J_\xi$ and $J_\eta$, and hence $J_\alpha$. Some example plots of \eqref{ellipseareaeq} against $t$ are given in Figure \ref{JArea}. We see that these curves have the approximate form of $\sin^2(t)$ which dips down below the axis to give two roots as expected from the previous section; these two roots $t=R_1$ and $t=R_2$ are the first conjugate points corresponding to the vanishing of two distinct Jacobi fields. We note there are some directions for which the function in Figure \ref{JArea} touches the axis at a double root, these are the `umbilic directions' mentioned in the previous section.

With the values of $R_1$ and $R_2$ found, we may recover the $(\theta,\phi,\psi)$ coordinates of these two first conjugate points along the geodesic emanating from $p$, and we do the same for a sphere's worth of directions in $T_p\mathcal{M}$; finally we may plot the two sheets of the conjugate locus in the coordinate space, see Figure \ref{sheets}. Firstly we observe the two sheets intersect one another in a complex way, so in Figure \ref{sheets}(a) we plot the sheets with low opacity so we can see inside. In Figure \ref{sheets}(b) and (c) we show the first sheet by itself, so we can more easily see the singularity structure: there is a closed rib that encircles the sheet, and two partial ribs that smooth out at either end. The second sheet is the same: one closed rib and two partial ribs. When the sheets are put together, we see the four partial ribs join to form a third closed rib, shared between the two sheets; the points where the partial ribs meet correspond to umbilic directions. We note that the base point $p$ chosen previously has three distinct principal curvatures, and that the picture in Figure \ref{sheets}(a) mirrors the form of the focal sets of a 2-ellipsoid with three distinct semi-axes, as claimed in the Introduction.

\begin{figure}[!htb]
    \centering
    \begin{tabular}{cc}

\begin{subfigure}{0.5\textwidth}
    
    \hspace{-2cm}
    \includegraphics[width=1.9\textwidth]{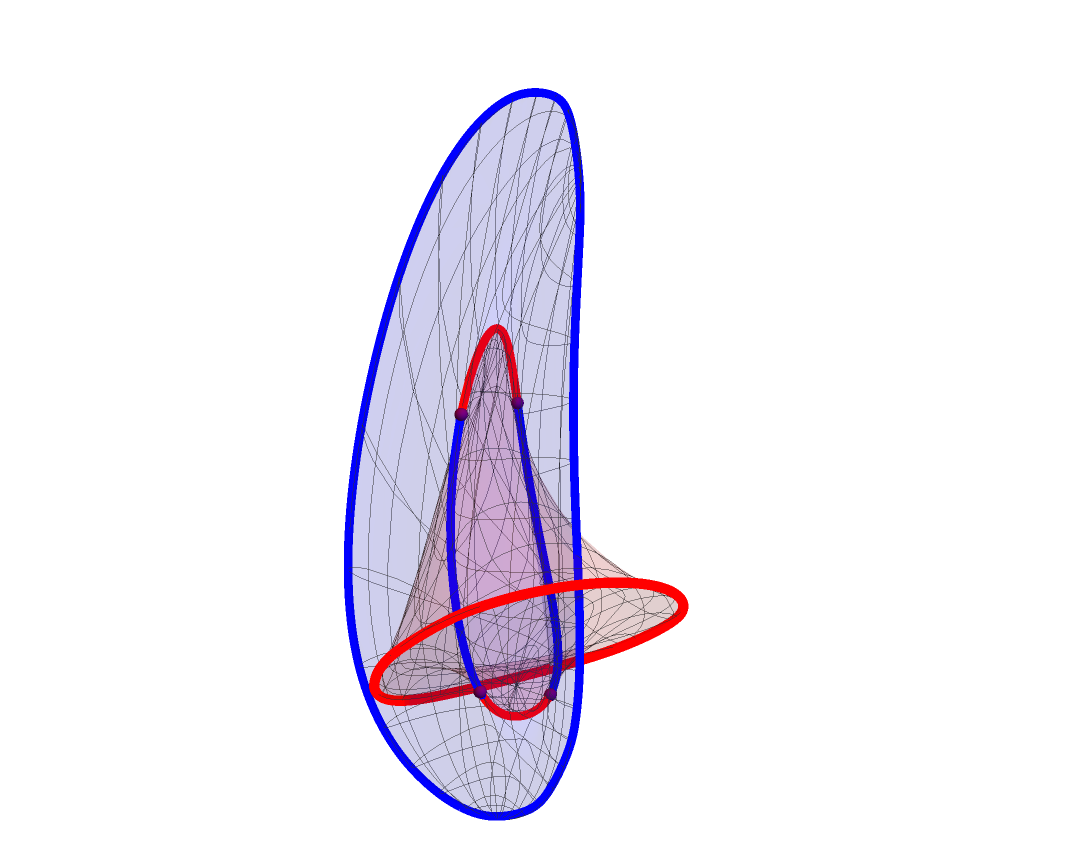}
    \caption{}
\end{subfigure}
    &
        \begin{tabular}{c}
        \smallskip
            \begin{subfigure}[t]{0.4\textwidth}
                \centering
                \includegraphics[width=0.9\textwidth]{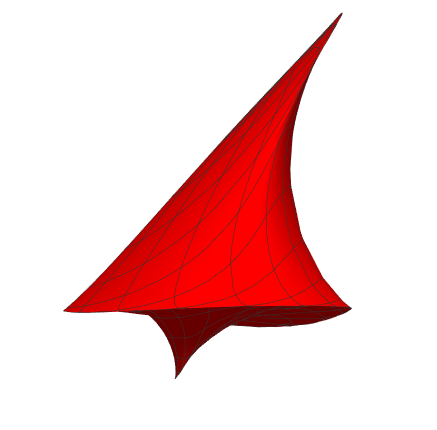}
                \caption{}
            \end{subfigure}\\
            \begin{subfigure}[t]{0.4\textwidth}
                \centering
                \includegraphics[width=0.9\textwidth]{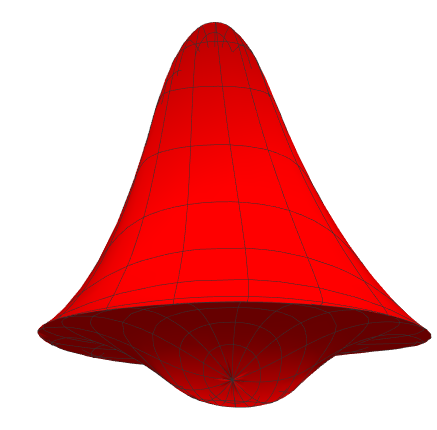}
                \caption{}
            \end{subfigure}
        \end{tabular}\\

    \end{tabular}
    \caption{The first conjugate locus in the quadraxial ellipsoid, for the semi-axes and base point given in the text. In (a) we see both sheets of the conjugate locus together (low-opacity). Shown thick are the 4 quarter-ribs of both sheets joining to form a closed curve, along with the complete rib from each sheet, and the umbilic directions shown in purple. In (b) and (c) we show the first sheet of the conjugate locus from two different angles; note the closed rib and two partial ribs.}\label{sheets}
\end{figure}

We now construct the previously described Jacobi coordinate system on the unit tangent-sphere given by the integral curves defined by the directions of collapsing Jacobi fields. At $t=0$ these direcitons are the initial velocity vectors of the vanishing Jacobi fields which correspond to the first and second conjugate points. As $t$ increases $J_\alpha$ grows and shrinks in these two directions, collapsing to a line segment at $t=R_1$ and $t=R_2$. To recover the corresponding $\alpha$-values, denoted $\alpha_1$ and $\alpha_2$, for which 
\begin{align*}
    \vert J_{\alpha}(t=R_i)\vert=0,\ i=1,2
\end{align*}
we compute the norm of $J_\alpha$ at both $t=R_1$ and $t=R_2$ over $\alpha \in [0,2\pi]$ and extract the roots. Hence, we may now construct the integral curves of $J'_{u}(0)$ and $J'_{v}(0)$ to give us the previously described Jacobi Coordinate system, shown in Figure \ref{JC0} on the left, for the base point mentioned previously (we let red and blue denote the $u$ and $v$ coordinate lines respectively). We note the similartiy between these Jacobi coordinate lines on the tangent-sphere and the curvature lines of the 2-ellipsoid, in particular we observe the coordinate lines as typically closed curves with a lemon structure at (hyperbolic) umbilic directions. 

Finally we find the ridge lines (the curves on the unit tangent-sphere which lift to the ribs under $X$), by detecting stationary points of $R_1$ and $R_2$ along the $u$ and $v$ coordinate lines respectively. That these ridge points must exist is immediate since for the Jacobi coordinate lines which are closed curves, the functions $R_1$ and $R_2$ restricted to them will generically have stationary points; as many maxima as minima and at least one of each.  Again we let red and blue denote $R_{1,u}=0$ and $R_{2,v}=0$ ridge lines respectively. We observe two closed ridge lines, one red and one blue, and one ridge line that passes from red to blue at the umbilic directions. These stationary points may be further separated into maxima and minima, allowing us to see some of the structure more clearly: the closed red ridge line represents minima of $R_1$, the closed blue line maxima of $R_2$. The broken line has maxima of $R_1$ meeting minima of $R_2$ at umbilic directions.

\begin{figure}
\begin{center}
{\includegraphics[width=0.5\textwidth]{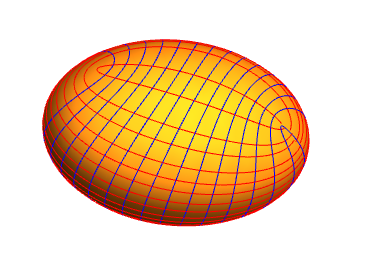}\includegraphics[width=0.4\textwidth]{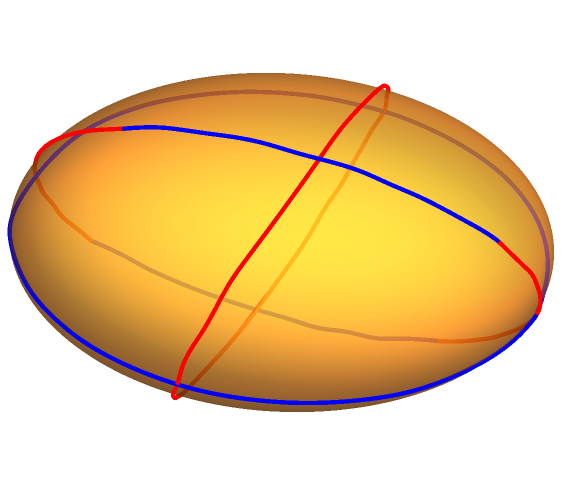}}
\caption{On the left the Jacobi coordinate lines on the unit tangent-sphere. The integral curves of the velocity vector of the first collapsing Jacobi fields are shown in red, those of the second are shown in blue. On the right are the ridge lines: we see two closed curves corresponding to the closed ribs on each sheet of the conjugate locus, and four partial ridges, corresponding to the four partial ribs, which meet at the umbilic directions to form a closed curve.}\label{JC0}
\end{center}
\end{figure}

While the plots shown seem quite complex, nonetheless we conjecture that the rib structure on the quadraxial ellipsoid is the simplest possible, that is (in analogy with the 4 vertex theorem or vierspitzensatz) the conjugate locus of a generic base point in a compact convex 3-manifold must have at least three ribs.

\section{Conclusions}

In this paper we have constructed a special coordinate system on the unit tangent sphere which we claim is a natural choice for studying the geometry and singular structures of the conjugate locus; it enables us to view the preimages of the ribs of the conjugate locus, observe the local structure surrounding umbilic directions, and view the similarities between this Jacobi coordinate system and the curvature lines of lower dimensional manifolds. Building on this approach we then developed a novel formalism for detecting conjugate points by tracking the area of the ellipse generated by a family of Jacobi fields around a central geodesic, before investigating the 3-ellipsoid as a case study to demonstrate the results and methodology described.

Further experimentation with the simulations described in this paper would be of interest: choosing different basepoints (including non-generic and umbilic points) to see how the structure of the conjugate locus is affected and test further the conjecture of the Introduction relating to ellipsoids; but also to broaden the study to manifolds other than ellipsoids, such as 3-d versions of the spherical harmonics menitoned in \cite{TWspherical,TWbif}. Also alluded to previously, it would be interesting to see how the conjugate locus transforms as the base point is dragged around in the manifold, especially in terms of the bifurcations of the ridges and ribs (see \cite{porteous}).

We also note the symmetric coefficient matrix of \eqref{jacobiequations}. While there is much known about constant symmetric matrices and normal modes, there is not much in the literature about time-dependent symmetric matrices (but see for example \cite{symmnorm}); perhaps the eigenvectors of this matrix might provide a natural choice for $N(0),B(0)$. We also observe how the rotation of the $J_\alpha$ ellipse in the $\{N,B\}$ plane, or lack thereof, might be worth further attention.

Finally, the method for detecting conjugate points extends very naturally to higher dimensions; for example if $n=4$ we will now have an ellipsoid in the orthogonal complement of $\gamma'$ which collapses to a disc at conjugate points, but provides still an oriented volume that changes sign (like a perturbed $\sin^3t$). The potential limitations here come with trying to visualize the results since now the coordinate space will be of dimesion $4$, but sections might be the way forward (see \cite{jandr}).

\bibliographystyle{plain}

\bibliography{clbib}

\end{document}